\documentclass[11pt, amsfonts]{amsart}

\usepackage{amsmath,amssymb,amsthm}
\usepackage[all]{xy}
\usepackage[dvipdfm,colorlinks=true]{hyperref}

\numberwithin{equation}{section}

\SelectTips{eu}{12}

\newtheorem{theorem}{Theorem}[section]
\newtheorem{proposition}[theorem]{Proposition}
\newtheorem{lemma}[theorem]{Lemma}
\newtheorem{corollary}[theorem]{Corollary}

\theoremstyle{definition}

\theoremstyle{remark}
\newtheorem{remark}[theorem]{Remark}

\newcommand{\Z}{\mathbb{Z}}
\newcommand{\Q}{\mathbb{Q}}

\newcommand{\C}{\mathbb{C}}
\newcommand{\G}{\mathcal{G}}
\newcommand{\Sp}{\mathrm{Sp}}

\title[On the homotopy types of $\Sp(n)$ gauge groups]{On the homotopy types of $\Sp(n)$ gauge groups}

\author{Daisuke Kishimoto}
\address{Department of Mathematics, Kyoto University, Kyoto, 606-8502, Japan}
\email{kishi@math.kyoto-u.ac.jp}

\author{Akira Kono}
\address{Faculty of Science and Engineering, Doshisha University, Kyoto 610-0321, Japan}
\email{akono@mail.doshisha.ac.jp}


\begin{document}

\baselineskip.525cm

\maketitle

\begin{abstract}
Let $\G_{k,n}$ be the gauge group of the principal $\Sp(n)$-bundle over $S^4$ corresponding to $k\in\Z\cong\pi_3(\Sp(n))$. We refine the result of Sutherland on the homotopy types of $\G_{k,n}$ and relate it with the order of a certain Samelson product in $\Sp(n)$. Then we classify the $p$-local homotopy types of $\G_{k,n}$ for $(p-1)^2+1\ge 2n$.
\end{abstract}


\section{Introduction}

Let $G$ be a topological group and $P\to X$ be a principal $G$-bundle over a base space $X$. The gauge group of $P$, denoted $\G(P)$, is the topological group of automorphisms of $P$, where an automorphism of $P$ is a $G$-equivariant self-map of $P$ covering the identity map of $X$. For fixed $G$ and $X$, one has a collection of gauge groups $\G(P)$ as $P$ ranges over all principal $G$-bundles over $X$, and we will be concerned with the classification of homotopy types in it.

Let $G$ be a compact connected simple Lie group. Then there is a one-to-one correspondence between (isomorphism classes of) principal $G$-bundles over $S^4$ and $\pi_3(G)\cong\Z$. We denote by $\G_k(G)$ the gauge group of the bundle corresponding to $k\in\Z\cong\pi_3(G)$. Consider the classification of the homotopy type in the collection of gauge groups $\{\G_k(G)\}_{k\in\Z}$. The first classification was done by the second named author \cite{K} for $G=\mathrm{SU}(2)$, and since then, considerable effort has been made for the classification when $G$ is of law rank \cite{HK,HKK,HKKS,KKKT,KTT,K,T2,T3,T4}. Properties of gauge groups related with the classification of the homotopy types have also been intensively studied \cite{CS,KK,KKTs1,KKTs2,KT,T1}.

In this paper, we study the classification of the homotopy types of $\G_k(\Sp(n))$. Let $\G_{k,n}=\G_k(\Sp(n))$. We will first consider Sutherland's homotopy invariant for $\G_{k,n}$ \cite{S}: if $\G_{k,n}$ and $\G_{l,n}$ are homotopy equivalent, then $(k,n(2n+1))=(l,n(2n+1))$ for $n$ even and $(k,4n(2n+1))=(l,4n(2n+1))$ for $n$ odd. It seems that this invariant has ambiguity by 4 according to the parity of $n$, and we will refine Sutherland's result by removing this ambiguity.

\begin{theorem}
\label{main1}
If $\G_{k,n}$ and $\G_{l,n}$ are homotopy equivalent, then $(k,4n(2n+1))=(l,4n(2n+1))$.
\end{theorem}

As for explicit classification of $\G_{k,n}$, there are only two results for $n=1,2$: $\G_{k,1}$ and $\G_{l,1}$ are homotopy equivalent if and only if $(k,12)=(l,12)$ \cite{K}, and $\G_{k,2}$ and $\G_{l,2}$ are $p$-locally homotopy equivalent for any prime $p$ if and only if $(k,40)=(l,40)$ \cite{T2}. The key fact that was used to prove these classification is that $\G_k(G)$ is homotopy equivalent to the homotopy fiber of the map $G\to\Omega^3_0G$ which is the adjoint of the Samelson product $S^3\wedge G\to G$ of $k\in\Z\cong\pi_3(G)$ and the identity map of $G$. Actually, the integers 12 and 40 in the above classification are the order of this Samelson product for $G=\Sp(1),\Sp(2)$, respectively. We will next show that the integer $4n(2n+1)$ in Theorem \ref{main1} is equal to the order of a certain Samelson product in $\Sp(n)$.

We set notation to state the result. Let $\epsilon\colon S^3\to\Sp(n)$ be the bottom cell inclusion so that it generates $\pi_3(\Sp(n))\cong\Z$. Let $Q_n$ be the quasi-projective space of rank $n$ defined in \cite{J}. Then one has the inclusion $\iota_n\colon Q_n\to\Sp(n)$ such that the induced map in homology
\begin{equation}
\label{Q}
\Lambda(\widetilde{H}_*(Q_n))\to H_*(\Sp(n))
\end{equation}
is an isomorphism. We denote by $\langle\alpha,\beta\rangle$ the Samelson product of maps $\alpha,\beta$.

\begin{theorem}
\label{main2}
The order of the Samelson product $\langle\epsilon,\iota_n\rangle$ in $\Sp(n)$ is $4n(2n+1)$.
\end{theorem}

It is obvious that the order of the Samelson products $\langle\epsilon,1_{\Sp(n)}\rangle$ is no less than the order of $\langle\epsilon,\iota_n\rangle$. Although we do not know these orders are equal, it is proved in \cite{KKTs2} that if we localize at a large prime $p$, these orders are equal. Let $|g|$ denote the order of an element $g$ of a group. For an integers $a=p^rq$ with $(p,q)=1$, let $\nu_p(a)=p^r$.

\begin{corollary}
\label{1_G}
If $(p-1)^2+1\ge 2n$, then $\nu_p(|\langle\epsilon,1_{\Sp(n)}\rangle|)=\nu_p(4n(2n+1))$.
\end{corollary}

\begin{remark}
The assumption in \cite[Theorem 1.4]{KKTs2}, which is needed to prove Corollary \ref{1_G}, is $(p-1)(p-2)+1\ge 2n$. But this assumption is actually too much and one can reduce it to $(p-1)^2+1\ge 2n$ as in Corollary \ref{1_G}. This refinement will be explained in Section 2.
\end{remark}

In \cite{KKTs2} the classification of the $p$-local homotopy types of $\G_{k,n}$ for a large prime $p$ is done in terms of the order of $\langle\epsilon,\iota_n\rangle$, by which one gets:

\begin{corollary}
\label{classification}
For $(p-1)^2+1\ge 2n$, $\G_{k,n}$ and $\G_{l,n}$ are $p$-locally homotopy equivalent if and only if $\nu_p((k,4n(2n+1)))=\nu_p((l,4n(2n+1)))$.
\end{corollary}

\begin{remark}
In \cite{T1} Theriault classified the $p$-local homotopy types of $\G_k(\mathrm{SU}(n))$ for $(p-1)^2+1\ge n$ by using Toda's map $\Sigma^2\C P^{n-1}\to\C P^n$ for Bott periodicity. It may be possible to prove Corollaries \ref{1_G} and \ref{classification} by modifying his method although Theorems \ref{main1} and \ref{main2} cannot. On the other hand, one can reprove Theriault's result by our method.
\end{remark}

As in \cite{F}, there is a $p$-local homotopy equivalence $B\mathrm{Spin}(2n+1)\simeq_{(p)}B\Sp(n)$ for any odd prime $p$, and we will see that this induces a $p$-local homotopy equivalence $\G_k(\mathrm{Spin}(2n+1))\simeq_{(p)}\G_{k,n}$ for any odd prime $p$. On the other hand, it is shown in \cite{KK} that a $p$-local homotopy equivalence $\mathrm{Spin}(2n+2)\simeq_{(p)}\mathrm{Spin}(2n+1)\times S^{2n+1}$ for any odd prime $p$ in \cite{BS} induces a $p$-local homotopy equivalence between $\G_k(\mathrm{Spin}(2n+2))$ and the product of $\G_k(\mathrm{Spin}(2n+1))$ and a certain space for any odd prime $p$. Combining these results with Corollary \ref{classification}, we get:

\begin{corollary}
\label{Spin}
For $(p-1)^2+1\ge 2n\ge 6$ and $\epsilon=1,2$, $\G_k(\mathrm{Spin}(2n+\epsilon))$ and $\G_l(\mathrm{Spin}(2n+\epsilon))$ are $p$-locally homotopy equivalent if and only if $\nu_p((k,4n(2n+1)))=\nu_p((l,4n(2n+1)))$x.
\end{corollary}

\emph{Acknowledgement:} The authors were partly supported by JSPS KAKENHI (No.\,17K05248 and No.\,15K04883).


\section{Odd primary homotopy types of gauge groups}

Let $\mathrm{map}(X,Y;f)$ be the path component of the mapping space $\mathrm{map}(X,Y)$ containing a map $f\colon X\to Y$. Let $G$ be a compact connected simple Lie group. In \cite{G,AB} it is shown that there is a homotopy equivalence
\begin{equation}
\label{map}
B\G_k(G)\simeq\mathrm{map}(S^4,BG;k\bar{\epsilon})
\end{equation}
where $\bar{\epsilon}$ corresponds to $1\in\Z\cong\pi_4(BG)$. So evaluating at the basepoint of $S^4$, one gets a homotopy fibration sequence
\begin{equation}
\label{fibration}
\G_k(G)\to G\xrightarrow{\partial_k}\Omega^3_0G\to B\G_k(G)\to BG.
\end{equation}
In particular, $\G_k(G)$ is homotopy equivalent to the homotopy fiber of $\partial_k$. Lang \cite{L} identified $\partial_k$ with a certain Samelson product in $G$. Let $\epsilon\colon S^3\to G$ be the adjoint of $\bar{\epsilon}$.

\begin{lemma}
The adjoint $S^3\wedge G\to G$ of $\partial_k$ is homotopic to the Samelson product $\langle k\epsilon,1_G\rangle$.
\end{lemma}

By linearity of Samelson products, we have $\langle k\epsilon,1_G\rangle=k\langle\epsilon,1_G\rangle$. We denote the $k$-th power map of $\Omega^3_0G$ by the same symbol $k$. Then we get:

\begin{corollary}
$\partial_k\simeq k\circ\partial_1$
\end{corollary}

Thus one sees that the order of the Samelson product $\langle\epsilon,1_G\rangle$ is connected to the classification of the homotopy types of $\G_k(G)$. It is shown in \cite{KKTs2} that, localized at a large prime, the calculation of the Samelson product $\langle\epsilon,1_G\rangle$ reduces drastically and the homotopy types of $\G_k(G)$ are classified in terms of the order of $\langle\epsilon,1_G\rangle$. We recall these results. Given a prime $p$, a space $A$ is called a homology generating space of an H-space $X$ if the following conditions hold:
\begin{enumerate}
\item $H_*(X;\Z/p)=\Lambda(x_1,\ldots,x_m)$;
\item there is a map $\iota\colon A\to X_{(p)}$ which induces the inclusion of a generating set in mod $p$ homology.
\end{enumerate}
An H-space $X$ is called retractible if it has a homology generating space $A$ and the map $\Sigma \iota\colon\Sigma A\to\Sigma X_{(p)}$ has a left homotopy inverse. It is proved in \cite{T} that if $(G,p)$ is in the following table, then $G_{(p)}$ is retractible, where we omit the cases $G=\mathrm{Spin}(2n)$ and $(G,p)=(\mathrm{G}_2,3)$.

\renewcommand{\arraystretch}{1.2}
\begin{table}[htbp]
\caption{Retractible Lie groups}
\label{retractible_Lie}
\centering
\begin{tabular}{l|l}
\hline
$\mathrm{SU}(n)$&$(p-1)^2+1\ge n$\\
$\Sp(n),\mathrm{Spin}(2n+1)$&$(p-1)^2+1\ge 2n$\\
$\mathrm{G}_2,\mathrm{F}_4,\mathrm{E}_6$&$p\ge 5$\\
$\mathrm{E}_7,\mathrm{E}_8$&$p\ge 7$\\
\hline
\end{tabular}
\end{table}

If $G$ has a homology generating space $A$ at a prime $p$, then the $p$-primary component of the order of $\langle\epsilon,\iota\rangle$ is obviously no less than that of $\langle\epsilon,1_G\rangle$. In \cite{KKTs2}, if $G$ is retractile in addition, then these two coincides. The assumption in \cite{KKTs2} for this result is stronger than retractibility but one can easily follow its proof to see that only retractibility is used. So we record this result here with a weaker assumption.

\begin{proposition}
\label{1_G-general}
If $(G,p)$ is in Table \ref{retractible_Lie}, then $\nu_p(|\langle\epsilon,1_G\rangle|)=\nu_p(|\langle\epsilon,\iota\rangle|)$.
\end{proposition}

Using this proposition, the following is proved in \cite{KKTs2}, where the assumption on the prime $p$ can be weakened as well.

\begin{theorem}
\label{p-local}
Suppose that $(G,p)$ is in Table \ref{retractible_Lie}. Then $\G_k(G)$ and $\G_l(G)$ are $p$-locally homotopy equivalent if and only if $\nu_p((k,|\langle\epsilon,\iota\rangle|))=\nu_p((l,|\langle\epsilon,\iota\rangle|))$.
\end{theorem}

\begin{proof}
[Proofs of Corollary \ref{1_G}]
Since \eqref{Q} is an isomorphism, $Q_n$ is a homology generating space of $\Sp(n)$ at any prime, and as in \cite{T}, $\Sp(n)$ is retractible with respect to $Q_n$ at the prime $p$. Then Corollary \ref{1_G} follows from Theorem \ref{main2} and Proposition \ref{1_G-general}.
\end{proof}

\begin{proof}
[Proof of Corollary \ref{classification}]
This follows from Corollary \ref{1_G} and Theorem \ref{p-local}.
\end{proof}

\begin{proof}
[Proof of Corollary \ref{Spin}]
We first consider the $p$-local homotopy type of $\G_k(\mathrm{Spin}(2n+1))$ for any odd prime $p$. By \cite{F}, we have $B\Sp(n)\simeq_{(p)}B\mathrm{Spin}(2n+1)$. Then it follows from \eqref{map} that $\G_k(\mathrm{Spin}(2n+1))\simeq_{(p)}\G_{k,n}$. Thus the result follows from Corollary \ref{classification}.

We next consider the $p$-local homotopy type of  $\G_k(\mathrm{Spin}(2n+2))$. Note that we are now assuming $p\ge 5$. Then it follows from \cite{KK} that there is a $p$-local homotopy equivalence
$$\G_k(\mathrm{Spin}(2n+2))\simeq_{(p)}\G_k(\mathrm{Spin}(2n+1))\times S^{2n+1}\times\Omega^4S^{2n+1}.$$
So the above case of $ \G_k(\mathrm{Spin}(2n+1))$ implies that $\G_k(\mathrm{Spin}(2n+2))\simeq_{(p)}\G_l(\mathrm{Spin}(2n+2))$ whenever $\nu_p((k,4n(2n+1)))=\nu_p(k,4n(2n+1))$. By \cite{S}, we have $\pi_{4n+1}(\G_{k,n})_{(p)}\cong\Z/\nu_p((k,4n(2n+1)))$. The order of $\pi_{4n+1}(S^{2n+1}\times\Omega^4S^{2n+1})_{(p)}$ is finite, say $M$, implying that the order of $\pi_{4n+1}(\G_k(\mathrm{Spin}(2n+2)))_{(p)}\cong\pi_{4n+1}(\G_{k,n}\times S^{2n+1}\times\Omega^4S^{2n+1})_{(p)}$ is $M\nu_p((k,4n(2n+1)))$. Thus we get that $\nu_p((k,4n(2n+1)))=\nu_p((l,4n(2n+1)))$ whenever $\G_k(\mathrm{Spin}(2n+2))\simeq_{(p)}\G_l(\mathrm{Spin}(2n+2))$, completing the proof.
\end{proof}


\section{Unstable $KSp$-theory}

If a space $Z$ is low dimensional, then the homotopy set $[Z,\mathrm{U}(n)]$ is isomorphic to $\widetilde{K}^{-1}(Z)$. So we call $[Z,\mathrm{U}(n)]$ unstable $K$-theory. In \cite{HK}, for $\dim Z\le 2n$, a computing method of $[Z,\mathrm{U}(n)]$ is given by comparing it with $\widetilde{K}^{-1}(Z)$. We call the homotopy set $[Z,\Sp(n)]$ unstable $KSp$-theory as well, and Nagao \cite{N} considered the analogous computing method of unstable $KSp$-theory. We will use Nagao's method to calculate Samelson products in $\Sp(n)$, so we recall it here. 

The cohomology of $B\Sp(n)$ and $\Sp(n)$ are given by
$$H^*(B\Sp(n))=\Z[q_1,\ldots,q_n],\quad H^*(\Sp(n))=\Lambda(x_3,\ldots,,x_{4n-1})$$
where $q_i$ is the $i$-th symplectic Pontrjagin class and $x_{4i-1}=\sigma(q_i)$ for the cohomology suspension $\sigma$. Let $X_n=\Sp(\infty)/\Sp(n)$. By an easy inspection, one sees that $H^*(X_n)=\Lambda(\bar{x}_{4n+3},\bar{x}_{4n+7},\ldots)$ for $\pi^*(\bar{x}_{4i-1})=x_{4i-1}$, where $\pi\colon\Sp(\infty)\to X_n$ is the projection. Then we get that $\Omega X_n$ is $(4n+1)$-connected and $H^{4n+2}(\Omega X_n)=\Z\{a_{4n+2}\}$, where $\sigma(\bar{x}_{4n+3})=a_{4n+2}$ and $R\{z_1,z_2,\ldots\}$ means the free $R$-module with a basis $\{z_1,z_2,\ldots\}$. In particular, the map $a_{4n+2}\colon\Omega X_n\to K(\Z,4n+2)$ is a loop map and is a $(4n+3)$-equivalence. So if $\dim\,Z\le 4n+2$, the map $(a_{4n+2})_*\colon[Z,\Omega X_n]\to H^{4n+2}(Z)$ is an isomorphism of groups. Moreover, it is shown in \cite{N} that the composite 
$$\widetilde{KSp}{}^{-2}(Z)=[Z,\Omega\Sp(\infty)]\xrightarrow{(\Omega\pi)_*}[Z,\Omega X_n]\xrightarrow{(a_{4n+2})_*}H^{4n+2}(Z)$$ 
is given by $(-1)^{n+1}(2n+1)!\mathrm{ch}_{4n+2}(u^{-1}c'(\xi))$ for $\xi\in\widetilde{KSp}{}^{-2}(Z)$, where $\mathrm{ch}_{k}$ denotes the $2k$-dimensional part of the Chern character, $u$ is a generator of $\widetilde{K}(S^2)\cong\Z$, and $c'\colon KSp\to K$ is the complexification. Now we apply $[Z,-]$ to a homotopy fibration sequence $\Omega\Sp(\infty)\to\Omega X_n\to \Sp(n)\to \Sp(\infty)$ and get an exact sequence of groups 
$$\widetilde{KSp}{}^{-2}(Z)\to[Z,\Omega X_n]\to[Z,\Sp(n)]\to\widetilde{KSp}{}^{-1}(Z).$$
Then by the above identification of $[Z,\Omega X_n]$, Nagao \cite{N} got:

\begin{theorem}
\label{unstable-K}
If $Z$ is a CW-complex of dimension $\le 4n+2$, then there is an exact sequence of groups
$$\widetilde{KSp}{}^{-2}(Z)\xrightarrow{\Phi}H^{4n+2}(Z)\to[Z,\Sp(n)]\to\widetilde{KSp}{}^{-1}(Z)$$
such that for $\xi\in\widetilde{KSp}{}^{-2}(Z)$,
$$\Phi(\xi)=(-1)^{n+1}(2n+1)!\mathrm{ch}_{4n+2}(u^{-1}c'(\xi)).$$
\end{theorem}

This is also useful to compute the Samelson products in $\Sp(n)$ as follows. Let $\gamma\colon\Sp(n)\wedge\Sp(n)\to\Sp(n)$ be the reduced commutator map. Since $\Sp(\infty)$ is homotopy commutative, the composite $\Sp(n)\wedge\Sp(n)\xrightarrow{\gamma}\Sp(n)\to\Sp(\infty)$ is null homotopic. Then since there is a homotopy fibration $\Omega X_n\to\Sp(n)\to\Sp(\infty)$, $\gamma$ lifts to a map $\widetilde{\gamma}\colon\Sp(n)\wedge\Sp(n)\to\Omega X_n$. In \cite{N}, a specific lift is contstructed as:

\begin{proposition}
\label{lift}
There is a lift $\widetilde{\gamma}\colon\Sp(n)\wedge\Sp(n)\to\Omega X_n$ of $\gamma$ satisfying
$$\widetilde{\gamma}^*(a_{4n+2})=\sum_{i+j=n+1}x_{4i-1}\otimes x_{4j-1}.$$
\end{proposition}

Thus by Theorem \ref{unstable-K}, one gets:

\begin{corollary}
\label{Samelson}
Let $A,B$ be CW-complexes such that $\dim A+\dim B\le 4n+2$. The order of the Samelson product of maps $\alpha\colon A\to\Sp(n)$ and $\beta\colon B\to\Sp(n)$ is equal to the order of
$$\sum_{i+j=n+1}\alpha^*(x_{4i-1})\otimes\beta^*(x_{4j-1})$$
in the cokernel of the map $\Phi\colon\widetilde{KSp}{}^{-2}(A\wedge B)\to H^{4n+2}(A\wedge B)$ of Theorem \ref{unstable-K}
\end{corollary}

The following data of $\widetilde{K}{}^*(Q_n)$ and $\widetilde{KSp}{}^*(Q_n)$ will be used to apply the above results to our case. Put $y_{4j-1}=\iota_n^*(x_{4j-1})$ for the inclusion $\iota_n\colon Q_n\to\Sp(n)$. Then we have $H^*(Q_n)=\Z\{y_3,\ldots,y_{4n-1}\}$. Let $\theta_1\colon\Sigma Q_2\to B\Sp(\infty)$ be the inclusion and $\theta_2$ be the composite of the pinch map onto the top cell $\Sigma Q_2\to S^8$ and a generator of $\pi_8(B\Sp(\infty))\cong\Z$. Then it follows that
\begin{equation}
\label{ch1}
\mathrm{ch}(c'(\theta_1))=\Sigma y_3-\frac{1}{6}\Sigma y_7,\quad\mathrm{ch}(c'(\theta_2))=2\Sigma y_7.
\end{equation}
Let $\rho_1=q(u^2c'(\theta_1))\in\widetilde{KSp}(\Sigma^5Q_2)$, where $q\colon K\to KSp$ is the quaternionization. Let $\rho_2\in\widetilde{KSp}(\Sigma^5Q_2)$ be the composite of the pinch map to the top cell $\Sigma^5Q_2\to S^{12}$ and a generator of $\pi_{12}(B\Sp(\infty))\cong\Z$. Then we have
$$\mathrm{ch}(c'(\rho_1))=2\Sigma^5 y_3+\frac{1}{3}\Sigma^5 y_7,\quad\mathrm{ch}(c'(\rho_2))=\Sigma^5 y_7.$$

\begin{lemma}
\label{KSp1}
$\widetilde{KSp}(\Sigma^iQ_2)=\begin{cases}\Z\{\theta_1,\theta_2\}&i=1\\\Z\{\rho_1,\rho_2\}&i=5\\0&i\equiv 0\mod 4\end{cases}$
\end{lemma}

\begin{proof}
A homotopy cofibration $S^4\to\Sigma Q_2\to S^8$ induces a commutative diagram with exact rows:
$$\xymatrix{0\ar[r]&\widetilde{KSp}(S^8)\ar[r]\ar[d]^{c'=2}&\widetilde{KSp}(\Sigma Q_2)\ar[r]\ar[d]^{c'}&\widetilde{KSp}(S^4)\ar[r]\ar[d]^{c'=1}&0\\
0\ar[r]&\widetilde{K}(S^8)\ar[r]&\widetilde{K}(\Sigma Q_2)\ar[r]&\widetilde{K}(S^4)\ar[r]&0}$$
Then we get the first equality by \eqref{ch1} and $\widetilde{KSp}(S^{4m})\cong\Z$. The remaining equalities are seen by the same argument.
\end{proof}

The complexification $c'\colon B\Sp(\infty)\to B\mathrm{U}(\infty)$ restricts to a map $\Sigma Q_n\to\Sigma^2\C P^{2n-1}$ which we denote by the same symbol $c'$. Let $\eta\in\widetilde{K}(\C P^{2n-1})$ be the Hopf bundle minus the trivial line bundle, and put $\xi_i=(c')^*(u\eta^i)\in\widetilde{K}(\Sigma Q_n)$. Then we have
$$\mathrm{ch}(c'(\xi_i))\equiv \Sigma x_{4i-1}\mod(\Sigma x_{4j-1}\,\vert\,j>i).$$
Thus by the skeletal argument analogous to the proof of Proposition \ref{KSp1}, one gets:

\begin{lemma}
\label{K}
$\widetilde{K}(\Sigma Q_n)=\Z\{\xi_1,\ldots,\xi_n\}$
\end{lemma}

Quite similarly, one gets the following as well.

\begin{proposition}
\label{KSp2}
$\widetilde{KSp}(\Sigma^5Q_n)=\Z\{\zeta_1,\ldots,\zeta_n\}$ such that $\zeta_1=q(u^2\xi_1)$ and 
$$\mathrm{ch}(c'(\zeta_i))\equiv\epsilon_i\Sigma^5y_{4i-1}\mod(\Sigma^5y_{4j-1}\,\vert\,j>i)$$
 for $i>1$, where $\epsilon_i=1$ for $i$ even and $\epsilon_i=2$ for $i$ odd.
\end{proposition}


\section{Proofs of the main theorems}

To prove Theorem \ref{main1}, we need several lemmas. Let $\widetilde{\partial}_k\colon\Sp(n)\to\Omega^4 X_n$ be the adjoint of the map $\widetilde{\gamma}\circ(\epsilon\wedge 1_{\Sp(n)})\colon S^3\wedge\Sp(n)\to\Omega X_n$, where $\widetilde{\gamma}$ is as in Proposition \ref{lift}. Then $\widetilde{\partial}_k$ is a lift of $\partial_k$, so by \eqref{fibration} and Theorem \ref{unstable-K}, we get the following commutative diagram with exact columns and rows, where $\delta_k=(a_{4n+2}\circ\tilde{\partial}_k)_*$.
\begin{equation}
\label{diagram}
\xymatrix{&\widetilde{KSp}{}^{-2}(\Sigma^{4n-5}Q_2)\ar[d]^\Phi\\
\widetilde{KSp}{}^{-1}(\Sigma^{4n-8}Q_2)\ar[r]^{\delta_k}\ar@{=}[d]&H^{4n+2}(\Sigma^{4n-5}Q_2)\ar[d]\\
\widetilde{KSp}{}^{-1}(\Sigma^{4n-8}Q_2)\ar[r]^{(\partial_k)_*}&[\Sigma^{4n-5}Q_2,\Sp(n)]\ar[r]\ar[d]&[\Sigma^{4n-8}Q_2,B\G_{k,n}]\ar[r]&\widetilde{KSp}(\Sigma^{4n-8}Q_2)\\
&\widetilde{KSp}{}^{-1}(\Sigma^{4n-5}Q_2)}
\end{equation}

\begin{lemma}
\label{coker}
$[\Sigma^{4n-8}Q_2,B\G_{k,n}]\cong\mathrm{Coker}\,(\partial_k)_*$
\end{lemma}

\begin{proof}

By Lemma \ref{KSp1}, one has $\widetilde{KSp}(\Sigma^{4n-8}Q_2)=0$, so the lemma follows from \eqref{diagram}.
\end{proof}

\begin{lemma}
\label{[Q,Sp]}
 $[\Sigma^{4n-5}Q_2,\Sp(n)]\cong\Z/\frac{(2n+1)!}{3}$ for $n$ even.
\end{lemma}

\begin{proof}
Since $\widetilde{KSp}{}^{-1}(\Sigma^{4n-5}Q_2)=0$ by Lemma \ref{KSp1}, we get $[\Sigma^{4n-5}Q_2,\Sp(n)]\cong\mathrm{Coker}\,\Phi$ by \eqref{diagram}. Since $n$ is even, $\widetilde{KSp}{}^{-2}(\Sigma^{4n-5}Q_2)\cong\widetilde{KSp}(\Sigma^5Q_2)$. Then it follows from Theorem \ref{unstable-K} and Lemma \ref{KSp1} that $\mathrm{Im}\,\Phi=\Z\{\frac{(2n+1)!}{3}\Sigma^{2n-5}y_7\}$. Thus for $H^{4n+2}(\Sigma^{4n-5}Q_2)=\Z\{\Sigma^{4n-5}y_7\}$, the proof is done.
\end{proof}

\begin{lemma}
\label{Im_d}
$\mathrm{Im}\,(\partial_k)_*\cong\Z/\frac{(2n+1)!}{3(k,4n(2n+1))}$ for $n$ even.
\end{lemma}

\begin{proof}
Since $\widetilde{KSp}{}^{-1}(\Sigma^{4n-5}Q_2)=0$ by Lemma \ref{KSp1}, we have $\mathrm{Im}\,(\partial_k)_*=\mathrm{Im}\,\delta_k/\mathrm{Im}\,\Phi$ by \eqref{diagram}. We calculate $\mathrm{Im}\,\delta_k$, where $\mathrm{Im}\,\Phi$ has been already calculated in the proof of Lemma \ref{[Q,Sp]}. Let $\widehat{\alpha}\colon\Sigma^{4n-8}Q_2\to\Sp(\infty)$ be the adjoint of $\alpha\in\widetilde{KSp}(\Sigma^{4n-7}Q_2)$. By definition, we have $\delta_k(\alpha)=k\Sigma^3\widehat{\alpha}^*(x_{4n-1})$, so we calculate $\widehat{\alpha}^*(x_{4n-1})$. Let $\mathrm{ch}(c'(\alpha))=a\Sigma^{4n-7}y_3+b\Sigma^{4n-7}y_7$ for $a,b\in\Q$. By the Newton formula, $\mathrm{ch}_{4n}=-\frac{1}{(2n-1)!}c_{2n}$+decomposables, implying that
$(-1)^n\alpha^*(q_n)=(c'\circ\alpha)^*(c_{2n})=-b(2n-1)!\Sigma^{4n-7}y_7$. Then by taking the adjoint, we get $\widehat{\alpha}^*(x_{4n-1})=(-1)^{n+1}b(2n-1)!\Sigma^{4n-8}y_7$. Since $n$ is even, we have $\widetilde{KSp}(\Sigma^{4n-7}Q_2)\cong\widetilde{KSp}(\Sigma Q_2)$. Thus by Lemma \ref{KSp1}, we obtain $\mathrm{Im}\,\delta_k=\Z\{\frac{k(2n-1)!}{6}\Sigma^{4n-5}y_7\}$. Therefore the proof is completed.
\end{proof}

\begin{lemma}
\label{lem_main1}
If $n$ is even and $n>2$, then $[\Sigma^{4n-8}Q_2,B\G_{k,n}]\cong\Z/(k,4n(2n+1))$. 
\end{lemma}

\begin{proof}
Combine Lemmas \ref{coker}, \ref{[Q,Sp]} and \ref{Im_d}.
\end{proof}

\begin{proof}
[Proof of Theorem \ref{main1}]
By the result of Sutherland \cite{S} mentioned above, it is sufficient to prove the theorem for $n$ even. When $n=2$, the result of Theriault \cite{T2} mentioned above implies the theorem. Assume that $n>2$ and $\G_{k,n}\simeq\G_{l,n}$. Then since $[\Sigma^{4n-8}Q_2,B\G_{m,n}]\cong[\Sigma^{4n-9}Q_2,\G_{m,n}]$ for any $m$,  we have $[\Sigma^{4n-8}Q_2,B\G_{k,n}]\cong[\Sigma^{4n-8}Q_2,B\G_{l,n}]$, so the theorem follows from Lemma \ref{lem_main1}.
\end{proof}

\begin{proof}
[Proof of Theorem \ref{main2}]
For $\dim\Sigma^3Q_n=4n+2$, we apply Corollary \ref{Samelson} to the Samelson product $\langle\epsilon,\iota_n\rangle$ in $\Sp(n)$. Then for $\sum_{i+j=n+1}\epsilon^*(x_{4i-1})\otimes\iota_n^*(x_{4j-1})=\Sigma^3y_{4n-1}$, it is sufficient to show that the image of $\Phi\colon\widetilde{KSp}{}^{-2}(\Sigma^3Q_n)\to H^{4n+2}(\Sigma^3Q_n)$  is generated by $4n(2n+1)\Sigma^3y_{4n-1}$. For $\zeta_1\in\widetilde{KSp}(\Sigma^5Q_n)$ of Lemma \ref{KSp2}, we have $\mathrm{ch}_{4n+2}(u^{-1}c'(\zeta_1))=\mathrm{ch}_{4n+2}((1+t)(u\xi_1))=\frac{2}{(2n-1)!}\Sigma^3y_{4n-1}$, so $4n(2n+1)\Sigma^3y_{4n-1}\in\mathrm{Im}\,\Phi$, where $t\colon K\to K$ is the complex conjugation. On the other hand, by Lemmas \ref{K} and \ref{KSp2}, $c'(\widetilde{KSp}(\Sigma^5Q_n))$ is included in $\Z\{c'(\zeta_1),u^2\xi_2,\ldots,u^2\xi_n\}\subset\widetilde{K}(\Sigma^5Q_n)$. By definition, we have
$$\mathrm{ch}_{4n+2}(u\wedge\xi_k)=\sum_{\substack{r_1+\cdots+r_k=2n-1\\r_1\ge 1,\ldots,r_k\ge 1}}\frac{(2n-1)!}{r_1!\cdots r_k!}\cdot\frac{1}{(2r_1-1)!\cdots(2r_k-1)!}\Sigma^3y_{4n-1}.$$
For $k\ge 2$, the coefficients of $(2n+1)!\mathrm{ch}_{4n+2}(u\xi_k)$ are divisible by $4n(2n+1)$. Then $\mathrm{Im}\,\Phi$ is included in the submodule generated by $4n(2n+1)\Sigma^3y_{4n-1}$. Thus we obtain that $\mathrm{Im}\,\Phi$ is generated by $4n(2n+1)\Sigma^3y_{4n-1}$ as desired. Therefore the proof is completed.
\end{proof}


\end{document}